\documentclass[12pt,leqno]{article}
\usepackage{hyperref}
\usepackage{soul}
\usepackage[doc,com]{optional}
\usepackage[doc]{optional}
\usepackage{tikz}
\usepackage{srcltx}
\usepackage{xcolor}
\usepackage{exscale,relsize}
\usepackage{amsmath}
\usepackage{amsfonts}
\usepackage{amssymb}
\usepackage{calc}
\usepackage{theorem}
\usepackage{pifont}      
\usepackage{graphicx}

\newcommand{\scal}[2]{\langle{{#1},{#2}}\rangle}

\newcommand{\RR}{\ensuremath{\mathbb R}}

\newcommand{\RX}{\ensuremath{\,\left]-\infty,+\infty\right]}}
\newcommand{\RXX}{\ensuremath{\,\left[-\infty,+\infty\right]}}

\newcommand{\NN}{\ensuremath{\mathbb N}}

\newcommand{\thalb}{\ensuremath{\tfrac{1}{2}}}

\newcommand{\menge}[2]{\big\{{#1} \mid {#2}\big\}}

\newcommand{\To}{\ensuremath{\rightrightarrows}}

\newcommand{\pos}{\operatorname{pos}}

\newcommand{\dom}{\ensuremath{\operatorname{dom}}}

\newcommand{\gra}{\ensuremath{\operatorname{gra}}}

\newcommand{\inte}{\ensuremath{\operatorname{int}}}

\newcommand{\ran}{\ensuremath{\operatorname{ran}}}

\newcommand{\pinf}{\ensuremath{+\infty}}

\renewcommand{\phi}{\ensuremath{\varphi}}

\newtheorem{theorem}{Theorem}[section]

\newtheorem{fact}[theorem]{Fact}

\newtheorem{openprob}[theorem]{Problem}
\theoremstyle{plain}{\theorembodyfont{\rmfamily}
}
\theoremstyle{plain}{\theorembodyfont{\rmfamily}
}
\theoremstyle{plain}{\theorembodyfont{\rmfamily}
}
\theoremstyle{plain}{\theorembodyfont{\rmfamily}
\newtheorem{example}[theorem]{Example}}
\theoremstyle{plain}{\theorembodyfont{\rmfamily}
\newtheorem{remark}[theorem]{Remark}}

\theoremstyle{plain}{\theorembodyfont{\rmfamily}
}

\begin{document}


\title{\sffamily{Monotone operators and ``bigger conjugate'' functions
 }}

\author{
Heinz H.\ Bauschke\thanks{Mathematics, Irving K.\ Barber School,
University of British Columbia, Kelowna, B.C. V1V 1V7, Canada.
E-mail: \texttt{heinz.bauschke@ubc.ca}.},\; Jonathan M.
Borwein\thanks{CARMA, University of Newcastle, Newcastle, New South
Wales 2308, Australia. E-mail:
\texttt{jonathan.borwein@newcastle.edu.au}.  Distinguished Professor
King Abdulaziz University, Jeddah. },\;
 Xianfu
Wang\thanks{Mathematics, Irving K.\ Barber School, University of British Columbia,
Kelowna, B.C. V1V 1V7, Canada. E-mail:
\texttt{shawn.wang@ubc.ca}.},\;
and Liangjin\
Yao\thanks{Mathematics, Irving K.\ Barber School, University of British Columbia,
Kelowna, B.C. V1V 1V7, Canada.
E-mail:  \texttt{ljinyao@interchange.ubc.ca}.}}

\date{August 12,  2011}
\maketitle

\begin{abstract} \noindent
We  study a
question posed by Stephen Simons in his 2008 monograph involving
``bigger conjugate'' (BC) functions and the partial infimal convolution.
As Simons demonstrated in his monograph, these function have been crucial
to the understanding and advancement of the state-of-the-art of harder problems in
monotone operator theory, especially the sum problem.

In this paper, we provide some tools for further analysis of BC--functions
which allow us to answer Simons' problem in the negative.
We are also able to refute a similar but much harder
conjecture which would have generalized a
classical result of Br\'{e}zis, Crandall and Pazy.
Our work also reinforces the importance of understanding unbounded
skew linear relations to construct monotone operators with unexpected
properties.
\end{abstract}

\noindent


\noindent

\section{Introduction}

Throughout this paper, we assume that $X$ is a real reflexive Banach space
with norm $\|\cdot\|$, that $X^*$ is the continuous dual of $X$, and
that $X$ and $X^*$ are paired by $\scal{\cdot}{\cdot}$.

 Let $A\colon
X\To X^*$ be a \emph{set-valued operator} (also known as a
multifunction) from $X$ to $X^*$, i.e., for every $x\in X$,
$Ax\subseteq X^*$, and let $\gra A:= \menge{(x,x^*)\in X\times
X^*}{x^*\in Ax}$ be the \emph{graph} of $A$.  The \emph{domain} of
$A$ is $\dom A:= \menge{x\in X}{Ax\neq\varnothing}$,  and $\ran
A:=A(X)$ for the \emph{range} of $A$. Recall that $A$ is
\emph{monotone} if
\begin{equation}
\scal{x-y}{x^*-y^*}\geq 0,\quad \forall (x,x^*)\in \gra A\;
\forall (y,y^*)\in\gra A,
\end{equation}
and \emph{maximally monotone} if $A$ is monotone and $A$ has
 no proper monotone extension
(in the sense of graph inclusion).
Let $S\subseteq X\times X^*$. We say $S$ is \emph{ a monotone set} if there exists a monotone
operator $A:X\rightrightarrows X^*$ such that $\gra A=S$, and $S$ is \emph{a maximally monotone set}
if there exists a maximally monotone
operator $A$ such that $\gra A=S$.
Let $A:X\rightrightarrows X^*$ be monotone and $(x,x^*)\in X\times X^*$.
 We say $(x,x^*)$ is \emph{monotonically related to}
$\gra A$ if
\begin{align*}
\langle x-y,x^*-y^*\rangle\geq0,\quad \forall (y,y^*)\in\gra
A.\end{align*}

 Maximally monotone operators have proven to be a potent class of
objects in modern Optimization and Analysis; see, e.g.,
\cite{Bor1,Bor2,Bor3}, the books \cite{BC2011,
BorVan,BurIus,ph,Si,Si2,RockWets,Zalinescu} and the references
therein.

We adopt standard notation used in these books especially
\cite[Chapter 2]{BorVan} and \cite{Bor1, Si, Si2}: Given a subset
$C$ of $X$,
$\inte C$ is the \emph{interior} of $C$,
$\overline{C}$ is the \emph{norm closure} of $C$. The \emph{support function} of $C$, written as $\sigma_C$,
 is defined by $\sigma_C(x^*):=\sup_{c\in C}\langle c,x^*\rangle$.
The \emph{indicator function} of $C$, written as $\iota_C$, is defined
at $x\in X$ by
\begin{align}
\iota_C (x):=\begin{cases}0,\,&\text{if $x\in C$;}\\
+\infty,\,&\text{otherwise}.\end{cases}\end{align}
For every $x\in X$, the \emph{normal cone operator} of $C$ at $x$
is defined by $N_C(x)= \menge{x^*\in
X^*}{\sup_{c\in C}\scal{c-x}{x^*}\leq 0}$, if $x\in C$; and $N_C(x)=\varnothing$,
if $x\notin C$. For $x,y\in X$, we set $\left[x,y\right]=\{tx+(1-t)y\mid 0\leq t\leq 1\}$.
 The \emph{closed unit
ball} is $B_X:=\menge{x\in X}{\|x\|\leq 1}$, and
$\NN:=\{1,2,3,\ldots\}$.

If $Z$ is a real  Banach space with dual $Z^*$ and a set $S\subseteq
Z$, we denote $S^\bot$ by $S^\bot := \{z^*\in Z^*\mid\langle
z^*,s\rangle= 0,\quad \forall s\in S\}$.
The \emph{adjoint} of an operator  $A$, written $A^*$, is defined by
\begin{equation*}
\gra A^* :=
\menge{(x,x^*)\in X\times X^*}{(x^*,-x)\in(\gra A)^{\bot}}.
\end{equation*}
We say $A$ is a \emph{linear relation} if $\gra A$ is a linear subspace.
 We say that $A$ is
\emph{skew}  if $\gra A \subseteq \gra (-A^*)$;
equivalently, if $\langle x,x^*\rangle=0,\; \forall (x,x^*)\in\gra A$.
Furthermore,
$A$ is \emph{symmetric} if $\gra A
\subseteq\gra A^*$; equivalently, if $\scal{x}{y^*}=\scal{y}{x^*}$,
$\forall (x,x^*),(y,y^*)\in\gra A$.

 Let $f\colon X\to \RX$. Then
$\dom f:= f^{-1}(\RR)$ is the \emph{domain} of $f$, and $f^*\colon
X^*\to\RXX\colon x^*\mapsto \sup_{x\in X}(\scal{x}{x^*}-f(x))$ is
the \emph{Fenchel conjugate} of $f$.  We say $f$ is proper if $\dom f\neq\varnothing$.
Let $f$ be proper. The \emph{subdifferential} of
$f$ is defined by
   $$\partial f\colon X\To X^*\colon
   x\mapsto \{x^*\in X^*\mid(\forall y\in
X)\; \scal{y-x}{x^*} + f(x)\leq f(y)\}.$$

\section{BC-functions}\label{sec:bc}

We now turn to the objects of the present paper: \emph{representative} and\emph{ BC-functions}.
Let $F:X\times X^*\rightarrow\RX$, and define $\pos F$ \cite{Si2}
 by  $$\pos F:=\big\{(x,x^*)\in X\times X^*\mid F(x,x^*)=\langle x,x^*\rangle\big\}.$$
 We say $F$ is a \emph{BC--function} (BC stands for
``bigger conjugate'') \cite{Si2} if $F$ is proper and
convex with
\begin{align} F^*(x^*,x)
 \geq F(x,x^*)\geq\langle x,x^*\rangle,\quad\forall(x,x^*)\in X\times X^*.
 \end{align}
The prototype for a BC function is the Fitzpatrick function \cite{Fitz88,Si2,BorVan}.

Let now $Y$ be another real Banach space. We set  $P_X: X\times Y\rightarrow
X\colon (x,y)\mapsto x$.
Let $F_1, F_2\colon X\times Y\rightarrow\RX$.
Then the \emph{partial inf-convolution}  $F_1\Box_2 F_2$
is the function defined on $X\times Y$ by
\begin{equation*}F_1\Box_2 F_2\colon
(x,y)\mapsto \inf_{v\in Y}
F_1(x,y-v)+F_2(x,v).
\end{equation*}

The importance of BC-functions associated with monotone operators
 is that along with appropriate partial convolutions, they provide
 the most powerful current method to establish the maximality of the
  sum of two maximally monotone operators \cite{Si2,BorVan}.
The two problems considered below are closely related to constructions
of maximally monotone operators as sums (see also Remark~\ref{r:final}).

The following question  was posed by S.\ Simons
\cite[Problem~34.7]{Si2}:

\begin{openprob}[Simons] \label{prob1}\emph{Let $F_1,F_2:X\times X^*\rightarrow\RX$ be
 proper lower semicontinuous and convex functions with $P_X\dom F_1\cap P_X\dom F_2\neq\varnothing$.
Assume that $F_1, F_2$ are BC--functions and  that there exists
an increasing function $j:\left[0,+\infty\right[\rightarrow
\left[0,+\infty\right[$ such that the implication
\begin{align*}&(x,x^*)\in\pos F_1, (y,y^*)\in\pos F_2, x\neq y\ \text{and}\
\langle x-y, y^*\rangle=\|x-y\|\cdot\|y^*\|\\
&\quad\Rightarrow\|y^*\|\leq
j\big(\|x\|+\|x^*+y^*\|+\|y\|+\|x-y\|\cdot\|y^*\|\big)
\end{align*}
holds.  Then, is it true that, for all $(z,z^*)\in X\times X^*$, there exists
$x^*\in X^*$ such that
\begin{align*}
F_1^*(x^*,z)+F_2^*(z^*-x^*,z)\leq (F_1\Box_2 F_2)^*(x^*,z)?
\end{align*}}
\end{openprob}

In Example~\ref{ex:donner} of  this paper,
we construct  a comprehensive negative answer to Problem~\ref{prob1}.
This in turn prompts another question:

\begin{openprob}\label{prob2} \emph{Let $F_1,F_2:X\times X^*\rightarrow\RX$ be
 proper lower semicontinuous and convex functions
  with $P_X\dom F_1\cap P_X\dom F_2\neq\varnothing$.
Assume that $F_1, F_2$ are BC--functions and  that there exists
an increasing function $j:\left[0,+\infty\right[\rightarrow
\left[0,+\infty\right[$ such that the implication
\begin{align*}&(x,x^*)\in\pos F_1, (y,y^*)\in\pos F_2, x\neq y\ \text{and}\
\langle x-y, y^*\rangle=\|x-y\|\cdot\|y^*\|\\
&\quad\Rightarrow\|y^*\|\leq
j\big(\|x\|+\|x^*+y^*\|+\|y\|+\|x-y\|\cdot\|y^*\|\big)
\end{align*}
holds.  Then, is it true that, for all $(z,z^*)\in X\times X^*$, there exists
$v^*\in X^*$ such that
\begin{align}
\label{Probcon}F_1^*(v^*,z)+F_2^*(z^*-v^*,z)\leq (F_1\Box_2 F_2)^*(z^*,z)?
\end{align}}
\end{openprob}

This is a quite reasonable question and somewhat harder to answer. An affirmative response to Problem \ref{prob2}
 would rederive Simons' theorem
(Fact~\ref{BCPCon}). Precisely, when the latter conjecture holds,
  we can deduce that $F:=F_1\Box_2 F_2$ is a BC-function.
It follows that   $\pos F$ (i.e., $M$ in Fact~\ref{BCPCon}) is
 a maximally monotone set; by Simons' result \cite[Theorem~21.4]{Si2}.
However, Example~\ref{FCPEX:3} shows that the  conjecture fails in
general.

We are now ready to set to work. The remainder of the paper is organized as follows.
In Section~\ref{s:aux}, we collect auxiliary results for future
reference and for the reader's convenience. Our main result
(Theorem~\ref{BCPON:2}) is established in Section~\ref{s:main}. In
Example~\ref{ex:donner},
we provide the promised negative answer to Problem~\ref{prob1}.
In Section~\ref{sec:4}, we
provide a negative answer to Problem~\ref{prob2}.

 \section{Auxiliary results}\label{s:aux}
\begin{fact}[Rockafellar]\label{SubMR}\emph{(See \cite[Theorem~A]{Rock702},
 \cite[Theorem~3.2.8]{Zalinescu}, \cite[Theorem~18.7]{Si2} or \cite[Theorem~2.1]{MSV})}
 Let $f:X\rightarrow\RX$ be a proper lower semicontinuous convex function.
Then $\partial f$ is maximally monotone.
\end{fact}

We now turn to  prerequisite results on Fitzpatrick functions,
monotone operators, and  linear relations.

\begin{fact}[Fitzpatrick]
\emph{(See {\cite[Corollary~3.9 and Proposition~4.2]{Fitz88}} and
\cite{Bor1,BorVan}.)} \label{f:Fitz} Let $A\colon X\To X^*$ be
maximally monotone, and set
\begin{equation}
F_A\colon X\times X^*\to\RX\colon
(x,x^*)\mapsto \sup_{(a,a^*)\in\gra A}
\big(\scal{x}{a^*}+\scal{a}{x^*}-\scal{a}{a^*}\big),
\end{equation}
which is the \emph{Fitzpatrick function} associated with $A$.
Then $F_A$ is a BC--function and $\pos F_A=\gra A$.
\end{fact}

\begin{fact}[Simons and Z\u{a}linescu]
\emph{(See \cite[Theorem~4.2]{SiZ} or \cite[Theorem~16.4(a)]{Si2}.)}\label{F4}
Let $Y$ be a real Banach space and $F_1, F_2\colon X\times Y \to \RX$ be proper,
lower semicontinuous, and convex. Assume that
for every $(x,y)\in X\times Y$,
\begin{equation*}(F_1\Box_2 F_2)(x,y)>-\infty
\end{equation*}
and that  $\bigcup_{\lambda>0} \lambda\left[P_X\dom F_1-P_X\dom F_2\right]$
is a closed subspace of $X$. Then for every $(x^*,y^*)\in X^*\times Y^*$,
\begin{equation*}
(F_1\Box_2 F_2)^*(x^*,y^*)=\min_{u^*\in X^*}
\left[F_1^*(x^*-u^*,y^*)+F_2^*(u^*,y^*)\right].
\end{equation*}\end{fact}

The following Simons' result generalizes
 the result of Br\'{e}zis, Crandall and Pazy \cite{BreCA}.
\begin{fact}[Simons]
\emph{(See \cite[Theorem~34.3]{Si2}.)}\label{BCPCon}
Let $F_1,F_2:X\times X^*\rightarrow\RX$ be
 proper lower semicontinuous and convex functions
  with $P_X\dom F_1\cap P_X\dom F_2\neq\varnothing$.
Assume that $F_1, F_2$ are BC--functions and  that there exists
an increasing function $j:\left[0,+\infty\right[\rightarrow
\left[0,+\infty\right[$ such that the implication
\begin{align*}&(x,x^*)\in\pos F_1, (y,y^*)\in\pos F_2, x\neq y\ \text{and}\
\langle x-y, y^*\rangle=\|x-y\|\cdot\|y^*\|\\
&\quad\Rightarrow\|y^*\|\leq
j\big(\|x\|+\|x^*+y^*\|+\|y\|+\|x-y\|\cdot\|y^*\|\big)
\end{align*}
holds. Then
$M:=\big\{(x, x^*+y^*)\mid (x,x^*)\in\pos F_1, (x,y^*)\in\pos F_2\big\}$
is a maximally monotone set.
\end{fact}

\section{Our main result}\label{s:main}

We start with two technical tools which relate Fitzpatrick
functions and skew operators. We  first give a  direct proof of the following result.

\begin{fact}
\emph{(See \cite[Corollary~5.9]{BBBRW}.)}
\label{f:referee04}
Let $C$ be a nonempty closed convex subset of $X$.
Then $F_{N_C}=\iota_C\oplus \iota^*_C$.
\end{fact}
\begin{proof}
Let $(x,x^*)\in X\times X^*$. Then we have
\begin{align}
F_{N_C}(x,x^*)&=\sup_{(c,c^*)\in\gra N_C}
\left[\langle x,c^*\rangle+\langle c,x^*\rangle-\langle c,c^*\rangle\right]\nonumber\\
&=\sup_{(c,c^*)\in\gra N_C, k\geq 0}
\left[\langle x,kc^*\rangle+\langle c,x^*\rangle-\langle c,kc^*\rangle\right]\nonumber\\
&=\sup_{(c,c^*)\in\gra N_C, k\geq 0}
\left[k(\langle x,c^*\rangle-
\langle c,c^*\rangle)+\langle c,x^*\rangle\right]\label{see:01}
\end{align}
By \eqref{see:01},
\begin{align}&(x,x^*)\in\dom F_{N_C}\Rightarrow
\sup_{(c,c^*)\in\gra N_C}\left[
\langle x,c^*\rangle-\langle c,c^*\rangle\right]\leq 0\nonumber\\
 &\Leftrightarrow
\inf_{(c,c^*)\in\gra N_C}\left[
-\langle x,c^*\rangle+\langle c,c^*\rangle\right]\geq 0\nonumber\\
&\Leftrightarrow
\inf_{(c,c^*)\in\gra N_C}\left[\langle c-x,c^*-0\rangle\right]\geq 0\nonumber\\
&\Leftrightarrow (x,0)\in \gra N_C\quad\text{(by Fact~\ref{SubMR})}\nonumber\\
&\Leftrightarrow x\in C.\label{see:02}\end{align}
Now assume $x\in C$. By \eqref{see:01},
 \begin{align}F_{N_C}(x,x^*)= \iota^*_C(x^*).\label{see:05}\end{align}
  Combine \eqref{see:02} and \eqref{see:05},
   $F_{N_C}=\iota_C\oplus \iota^*_C$.
\end{proof}

\begin{fact}\emph{(See \cite[Proposition~5.5]{BWY3}.)}\label{LeSK:a1}
Let $A\colon X\To X^*$ be a monotone linear relation
such that $\gra A\neq\varnothing$ and $\gra A$ is closed. Then
\begin{align}
 F_A^*(x^*,x)=\iota_{\gra A}(x,x^*)+\langle x,x^*\rangle,\
 \forall (x,x^*)\in X\times X^* .\label{Lesk:1}
\end{align}
\end{fact}

We are now ready to establish our main result.

\begin{theorem}
\label{BCPON:2}
Let $A: X\rightrightarrows X^{*}$ be a
 maximally monotone linear relation that is at most single-valued,
and let $C\neq\{0\}$ be a bounded closed and  convex subset of $X$
such that $\bigcup_{\lambda>0}\lambda\left[\dom A-C\right]$
is a closed subspace of $X$.
Let $j:\left[0,+\infty\right[\rightarrow
\left[0,+\infty\right[$ be an increasing function such that
 $j(\gamma)\geq\gamma$  for every $\gamma\in\left[0,+\infty\right[$.
Then the following hold.
\begin{enumerate}
\item\label{BCPONA:1}
$F_A$ and $F_{N_C}=\iota_C\oplus\sigma_C$ are BC-functions.

\item\label{BCPONA:3}
 $F_A^*(x^*,x)+F_{N_C}^*(y^*-x^*,x)=
 \iota_{\gra A\cap C\times X^*}(x,x^*)+\langle x, x^*\rangle+\sigma_C(y^*-x^*),
 \quad \forall (x,x^*,y^*)\in X\times X^*\times X^*$.
\item\label{BCPONA:2}
For every $(x,x^*)\in X\times X^*$,
\begin{equation}
\label{e:donner}
(F_A\Box_2 F_{N_C})^*(x^*,x) = \displaystyle \begin{cases}
\langle x,Ax\rangle + \sigma_C(x^*-Ax), &\text{if $x\in C\cap \dom
A$;}\\
\pinf, &\text{otherwise.}
\end{cases}
\end{equation}
%
\item \label{BCPONA:4}There exists $(z,z^*)\in X\times X^*$
 such that $z\in\dom A\cap C$ and
$\sigma_C(z^*-Az)>0$.
\item Assume that \label{BCPONA:04}$(z,z^*)\in X\times X^*$
 satisfies $z\in\dom A\cap C$ and
$\sigma_C(z^*-Az)>0$. Then
\begin{align}
F_A^*(x^*,z)+F_{N_C}^*(z^*-x^*,z)>(F_A\Box_2 F_{N_C})^*(x^*,z),
\quad\forall x^*\in X^*.\label{BCPONA:v0}
\end{align}
\item\label{BCPONA:5}Moreover,  assume that $X$ is a Hilbert space and $C=B_X$. Then
the implication
\begin{align}&(x,x^*)\in\pos F_A, (y,y^*)\in\pos F_{N_C}, x\neq y\ \text{and}\
\langle x-y, y^*\rangle=\|x-y\|\cdot\|y^*\|\nonumber\\
&\quad\Rightarrow\|y^*\|\leq \|x^*+y^*\|\leq j\big(\|x\|+\|x^*+y^*\|
+\|y\|+\|x-y\|\cdot\|y^*\|\big)\label{PONA:2}
\end{align}
 holds.
\end{enumerate}

\end{theorem}
\begin{proof}
\ref{BCPONA:1}:
Combine  Fact~\ref{f:referee04} and
Fact~\ref{f:Fitz}.

\ref{BCPONA:3}:
Let $(x,x^*,y^*)\in X\times X^*\times X^*$. Then by
Fact~\ref{LeSK:a1} and \ref{BCPONA:1}, we have
\begin{align*}
F_A^*(x^*,x)+F_{N_C}^*(y^*-x^*,x)&=\iota_{\gra A}(x,x^*)
+\langle x,x^*\rangle+(\iota^*_C\oplus\sigma^*_C)(y^*-x^*,x)\\
&=\iota_{\gra A}(x,x^*)+\langle x,x^*\rangle+\iota_C(x)+\sigma_C(y^*-x^*)\\
&=\iota_{\gra A\cap C\times X^*}(x,x^*)+\langle x,x^*\rangle+\sigma_C(y^*-x^*).
\end{align*}

\ref{BCPONA:2}:
By \cite[Lemma 5.8]{BWY3}, we have
\begin{align}\bigcup_{\lambda>0} \lambda
\big(P_{X}(\dom F_A)-P_{X}(\dom F_{N_C})\big)\
 \text{ is a closed subspace of $X$}.\label{Sppl:a01}\end{align}
Then for every $(x,x^*)\in X\times X^*$ and $u^*\in X^*$, by \ref{BCPONA:1},
\begin{equation*}
F_A (x,u^*)+ F_{N_C}(x,x^*-u^*)\geq \langle x,u^*\rangle+\langle
x,x^*-u^*\rangle=\langle x,x^*\rangle.
\end{equation*}
Hence
\begin{align}(F_A\Box_2F_{N_C})(x,x^*)\geq\langle x,x^*\rangle>-\infty.\label{Sppl:a1}
\end{align}

By \eqref{Sppl:a01}, \eqref{Sppl:a1}, Fact~\ref{F4},
and \ref{BCPONA:3}, for every $(x,x^*)\in X\times X^*$,
there exists $z^*\in X^*$ such that
\begin{align}
(F_A\Box_2F_{N_C})^* (x^*,x)
&=
\min_{y^*\in X^*}F^*_A ( y^*,x)+ F_{N_C}^*( x^*-y^*,x)\nonumber\\
&=\iota_{\gra A\cap C\times X^*}(x,z^*)+\langle
x,z^*\rangle+\sigma_C(x^*-z^*).
\end{align}
This implies \eqref{e:donner}.

\ref{BCPONA:4}: By the assumption, there exists $z\in\dom A\cap C$.
 Since $C\neq\{0\}$, there
exists $z^*\in X^*$ such that $\sigma_C(z^*-Az)>0$.

\ref{BCPONA:04}:
Let $x^*\in X^*$.
By the assumptions, \ref{BCPONA:2} and the boundedness of $C$, we have
\begin{align}
(F_A\Box_2 F_{N_C})^*(x^*,z)=\langle z,Az\rangle+\sigma_C( x^*-Az)<+\infty.
\label{BCPONA:v1}
\end{align}
We consider two cases.

\emph{Case 1}: $x^*\neq Az$.

Then $(z,x^*)\notin\gra A$ and so $\iota_{\gra A\cap C\times
X^*}(z,x^*)=\pinf$.
In view of  \ref{BCPONA:3} and \eqref{BCPONA:v1},  \eqref{BCPONA:v0} holds.

\emph{Case 2}: $x^*=Az$.

By \ref{BCPONA:3} and \eqref{BCPONA:v1},
we have
\begin{align*}
F_A^*(x^*,z)+F_{N_C}^*(z^*-x^*,z)=\langle z,Az\rangle+
\sigma_C(z^*-Az)&>\langle z,Az\rangle+0=\langle z,Az\rangle+\sigma_C(0)\\
&=(F_A\Box_2 F_{N_C})^*(x^*,z).
\end{align*}
Hence \eqref{BCPONA:v0} holds as well.

\ref{BCPONA:5}:
We start with a well known formula whose short proof
we include for completeness.
Let $x\in X$. Then
\begin{align}
N_{B_X}(x)&=\begin{cases}0,\;&\text{if}\; \|x\|<1;\\
\left[0,\infty\right[\cdot x,\;&\text{if}\; \|x\|=1;\\
\varnothing,\;&\text{otherwise}.\label{esee:4}\end{cases}
\end{align}
Clearly, $N_{B_X}(x)=0$ if $\|x\|<1$, and $N_{B_X}(x)=\varnothing$
 if $x\notin B_X$. Assume $\|x\|=1$.
Then
\begin{align*}
x^*\in N_{B_X}(x)&\Leftrightarrow
\|x^*\|=\|x^*\|\cdot\|x\|\geq\langle x^*,x\rangle\geq
\sup\langle x^*, B_X\rangle=\|x^*\|\\
&
\Leftrightarrow \langle x^*,x\rangle=\|x^*\|\cdot\|x\|\\
&
\Leftrightarrow x^*=\gamma x,\quad \gamma\geq0.\end{align*}
Hence \eqref{esee:4} holds.

Now let $(x,x^*)\in\pos F_A, (y,y^*)\in\pos F_{N_C}\ \text{and}\ x\neq y\ \text{be such that}\
\langle x-y, y^*\rangle=\|x-y\|\cdot\|y^*\|$. By Fact~\ref{f:Fitz},
\begin{align}
x^*=Ax \ \text{and}\ y^*\in N_{B_X}(y).\label{PONA:004}
\end{align}
Now we show that
\begin{align}
\|x^*+y^*\|\geq\|y^*\|.\label{PONA:3}
\end{align}
Clearly, \eqref{PONA:3} holds if $y^*=0$.
Thus, we assume that $y^*\neq0$. By \eqref{PONA:004} and \eqref{esee:4},
 there exists $\gamma_0>0$ such that
\begin{align}
y^*=\gamma_0 y,\label{PONA:4}
\end{align}where
\begin{align}
\|y\|=1 .\label{PONA:04}
\end{align}
Since $\langle x-y, y^*\rangle=\|x-y\|\cdot\|y^*\|$, we have
\begin{align}
y^*=\frac{\|y^*\|}{\|x-y\|}(x-y).\label{PONA:5}
\end{align}
We claim that
\begin{align}
x\neq0.\label{PONA:c4}
\end{align}
Suppose to the contrary that $x=0$.
Then by \eqref{PONA:5} and \eqref{PONA:04}, we have
$y^*=-\frac{\|y^*\|}{\|y\|}y=-\|y^*\|y$, which contradicts \eqref{PONA:4}.
Hence \eqref{PONA:c4} holds.

By
\eqref{PONA:4}, \eqref{PONA:5} and \eqref{PONA:c4}, we have
\begin{align}
\frac{x}{\|x\|}=\frac{y^*}{\|y^*\|}.\label{PONA:6}
\end{align}
Then \eqref{PONA:004} and the monotonicity of $A$ imply
\begin{align*}
\|x^*+y^*\|\geq\langle x^*+y^*, \frac{x}{\|x\|}\rangle
\geq\langle y^*, \frac{y^*}{\|y^*\|}\rangle=\|y^*\|.
\end{align*}
Therefore, \eqref{PONA:3} holds.

Then by the assumption, we have
\begin{align*}
j\big(\|x\|+\|x^*+y^*\|+\|y\|+\|x-y\|\cdot\|y^*\|\big)
&\geq j\big(\|x^*+y^*\|\big)\\
&\geq\|x^*+y^*\|\\
&\geq\|y^*\|.
\end{align*}
Hence \eqref{PONA:2} holds,
\end{proof}

We are now ready to exploit Theorem \ref{BCPON:2} to resolve Problem \ref{prob1}.

\begin{example}
\label{ex:donner}
Suppose that $X$ is a Hilbert space,
and let $A: X\rightrightarrows X^{*}$ be a
 maximally monotone linear relation that is at most single-valued,
and set $C=B_X$.
Let $j:\left[0,+\infty\right[
 \rightarrow\left[0,+\infty\right[$ be an increasing function such that
 $j(\gamma)\geq\gamma$ for every $\gamma\in\left[0,+\infty\right[$.
Then the following hold.
\begin{enumerate}
\item\label{BCREX:a1}Let $z^*\neq 0$. Then
\begin{align*}
F_{A}^*(x^*,0)+F_{N_C}^*(z^*-x^*,0)>(F_{A}
\Box_2 F_{N_{C}})^*(x^*,0),\quad\forall x^*\in X.
\end{align*}
\item
The implication
\begin{align*}&(x,x^*)\in\pos F_{A}, (y,y^*)\in
\pos F_{N_{C}}, x\neq y\ \text{and}\
\langle x-y, y^*\rangle=\|x-y\|\cdot\|y^*\|\nonumber\\
&\quad\Rightarrow\|y^*\|\leq\|x^*+y^*\|\leq
j\big(\|x\|+\|x^*+y^*\|+\|y\|+\|x-y\|\cdot\|y^*\|\big)
\end{align*}
 holds.
\end{enumerate}
\end{example}
\begin{proof}
Set $z=0$.
Then $Az=0$
$\Rightarrow$
$z^*-Az=z^*\neq 0$
$\Rightarrow$
$\sigma_C(z^*-Az)=\sigma_C(z^*)=\|z^*\|>0$.
Now apply Theorem~\ref{BCPON:2}\ref{BCPONA:04}\&\ref{BCPONA:5}.
\end{proof}

\begin{remark}
Example~\ref{ex:donner} yields
a negative answer to Simons' Problem \ref{prob1}
(\cite[Problem~34.7]{Si2}) for many linear relations --- including the rotation by 90 degrees in the plane.
\end{remark}

\section{Resolution of Problem \ref{prob2}}\label{sec:4}

We now move to the second problem. Its resolution depends on the following
fact concerning a maximally monotone operator on $\ell^2$,
the real Hilbert space of square-summable sequences.

\begin{fact}{\rm (See \cite[Propositions~3.5, 3.6 and 3.7 and Lemma~3.18]{BWY7}.)}
\label{FE:1}
Suppose that  $X=\ell^2$,  and that
$A:\ell^2\rightrightarrows \ell^2$ is given by
\begin{align}Ax:=\frac{\bigg(\sum_{i< n}x_{i}-\sum_{i> n}x_{i}\bigg)_{n\in\NN}}{2}
=\bigg(\sum_{i< n}x_{i}+\tfrac{1}{2}x_n\bigg)_{n\in\NN},
\quad \forall x=(x_n)_{n\in\NN}\in\dom A,\label{EL:1}\end{align}
where $\dom A:=\Big\{ x:=(x_n)_{n\in\NN}\in \ell^{2}\mid \sum_{i\geq 1}x_{i}=0,
 \bigg(\sum_{i\leq n}x_{i}\bigg)_{n\in\NN}\in\ell^2\Big\}$ and $\sum_{i<1}x_{i}:=0$.
Then
\begin{align}
\label{PF:a2}
A^*x= \bigg(\thalb x_n + \sum_{i> n}x_{i}\bigg)_{n\in\NN},
\end{align}
where
\begin{equation*}
x=(x_n)_{n\in\NN}\in\dom A^*=\bigg\{ x=(x_n)_{n\in\NN}\in \ell^{2}\;\; \bigg|\;\;
 \bigg(\sum_{i> n}x_{i}\bigg)_{n\in\NN}\in \ell^{2}\bigg\}.
\end{equation*}
Then $A$ provides an at most single-valued linear relation
such that the following hold.
\begin{enumerate}
 \item
 $A$\label{NEC:1} is maximally monotone and skew.
\item\label{NEC:2} $A^*$ is maximally monotone but not skew.
\item\label{NEC:4} $F_{A^*}^{*}(x^*,x)=F_{A^*}(x,x^*)
=\iota_{\gra A^*}(x,x^*)+\scal{x}{x^*},\quad \forall(x,x^*)\in X\times X$.
\item\label{NEC:5} $\langle A^*x, x\rangle=\tfrac{1}{2}s^2,
\quad \forall x=(x_n)_{n\in\NN}\in\dom A^*\ \text{with}\quad
s:=\sum_{i\geq1} x_i$.
\end{enumerate}
\end{fact}

We are now ready for the main construction of this section.

\begin{example}\label{FCPEX:3}
Suppose that $X$ and $A$ are as in Fact~\ref{FE:1}.
Set $e_1:=(1,0,\ldots,0,\ldots)$, i.e., there is a $1$
 in the first place and all  others entries are $0$, and $C:=\left[0,e_1\right]$.
 Let $j:\left[0,+\infty\right[
 \rightarrow\left[0,+\infty\right[$ be an increasing function such that
 $j(\gamma)\geq\tfrac{\gamma}{2}$ for every $\gamma\in\left[0,+\infty\right[$.
Then the following hold.
\begin{enumerate}
\item\label{BCPONA:E01} $F_{A^*}$ and $F_{N_C}=\iota_C\oplus\sigma_C$ are BC--functions.
\item\label{BCPONA:E1}
$(F_{A^*}\Box_2 F_{N_C})(x,x^*)=
\begin{cases}
\langle x,A^*x\rangle + \sigma_C(x^*-A^*x),&\text{if $x\in C$;}\\
\pinf, &\text{otherwise,}
\end{cases}
\quad \forall (x,x^*)\in X\times X^*$.
\item\label{BCREX:E2} Then
\begin{align*}
F^*_{A^*}(x^*,0)+F_{N_{C}}^*(A^*e_1-x^*,0)>
(F_{A^*}\Box_2 F_{N_{C}})^*(A^*e_1,0),\quad\forall x^*\in X.
\end{align*}
\item\label{BCREX:E3}
The implication
\begin{align*}&(x,x^*)\in\pos F_{N_{C}} , (y,y^*)\in\pos F_{A^*}, x\neq y\ \text{and}\
\langle x-y, y^*\rangle=\|x-y\|\cdot\|y^*\|\nonumber\\
&\quad\Rightarrow\|y^*\|\leq\tfrac{1}{2}\|y\|\leq
j\big(\|x\|+\|x^*+y^*\|+\|y\|+\|x-y\|\cdot\|y^*\|\big)
\end{align*}
 holds.
\item\label{BCREX:E4} $A^*+N_C$ is maximally monotone.

\end{enumerate}
\end{example}

\begin{proof}
\ref{BCPONA:E01}: Combine Fact~\ref{FE:1}\ref{NEC:2},
 Fact~\ref{f:Fitz} and Fact~\ref{f:referee04}.

\ref{BCPONA:E1}: Using Fact~\ref{FE:1}\ref{NEC:4},
we see that for every $(x,x^*)\in X\times X^*$,
\begin{align*}
(F_{A^*}\Box_2 F_{N_C})(x,x^*) &=
\inf_{y^*\in X^*}\iota_{\gra A^*}(x,y^*) +
\langle x,y^*\rangle + \iota_C(x) + \sigma_C(x^*-y^*)\\
&=\begin{cases}
\langle x,A^*x\rangle + \sigma_C(x^*-A^*x),&\text{if $x\in \dom A^*\cap C$;}\\
\pinf, &\text{otherwise,}
\end{cases}.
\end{align*}
The identity now follows since $C\subseteq\dom A^*$.

\ref{BCREX:E2}:  Let $x^*\in X$. Then by Fact~\ref{FE:1}\ref{NEC:4} we have
\begin{align}
F^*_{A^*}(x^*,0)+F_{N_{C}}^*(A^*e_1-x^*,0)
&=\iota_{\{0\}}(x^*) + \sigma_C(A^*e_1-x^*)\nonumber\\
&=\sigma_C(A^*e_1)+\iota_{\{0\}}(x^*)\nonumber\\
&=\sup_{t\in\left[0,1\right]}\big\{t\langle e_1, A^*e_1\rangle\big\}+
\iota_{\{0\}}(x^*)\nonumber\\
&=\langle e_1, A^*e_1\rangle+ \iota_{\{0\}}(x^*)\nonumber\\
&=\tfrac{1}{2}+\iota_{\{0\}}(x^*)\quad\text{
(by Fact~\ref{FE:1}\ref{NEC:5})}.\label{BCNC:1}\end{align}
On the other hand, by \ref{BCPONA:E1} and $C\subseteq\dom A^*$ by Fact~\ref{FE:1}, we have
\begin{align}
(F_{A^*}\Box_2 F_{N_{C}})^*(A^*e_1,0)
&=\sup_{x\in C, x^*\in X}\big\{\langle A^*e_1, x\rangle-
\langle x, A^*x\rangle-\sigma_C(x^*-A^*x)\big\}\nonumber\\
&\leq\sup_{x\in C, x^*\in X}\big\{\langle A^*e_1, x\rangle-
\langle x, A^*x\rangle\big\}\quad\text{(by $0\in C$)}\nonumber\\
&=\sup_{t\in\left[0,1\right]}\big\{t\langle A^*e_1,
 e_1\rangle-t^2\langle e_1, A^*e_1\rangle\big\}\nonumber\\
&=\tfrac{1}{4}\langle A^*e_1, e_1\rangle\nonumber\\
&=\tfrac{1}{8}\quad\text{(by Fact~\ref{FE:1}\ref{NEC:5})}\nonumber\\
&<F^*_{A^*}(x^*,0)+F_{N_{C}}^*(A^*e_1-x^*,0)\quad\text{(by \eqref{BCNC:1})}.\nonumber
\end{align}
Hence \ref{BCREX:E2} holds.

\ref{BCREX:E3}:
Let $(x,x^*)\in\pos F_{N_{C}} , (y,y^*)\in\pos F_{A^*},\
 \text{and}\ x\neq y\ \text{be such that}\
\langle x-y, y^*\rangle=\|x-y\|\cdot\|y^*\|$. By Fact~\ref{f:Fitz},
\begin{align}
x^*\in N_{C}(x)\ \text{and}\ y^*=A^*y .\label{PONCN:004}
\end{align}
Now we show
\begin{align}
\tfrac{1}{2}\|y\|\geq\|y^*\|.\label{PONACT:3}
\end{align}
Clearly, \eqref{PONACT:3} holds if $y^*=0$. Now assume that $y^*\neq0$. Then
by $\langle x-y, y^*\rangle=\|x-y\|\cdot\|y^*\|$ and
 $x\in C$, there exist $t_0\geq0$ and  $\gamma_0>0$ such that
\begin{align}
x=t_0 e_1 \ \text{and}\
y^*=\gamma_0(t_0e_1-y).\label{PONACT:4}
\end{align}
Write $y=(y_n)_{n\in\NN}$. By \eqref{PF:a2} and \eqref{PONACT:4},
we have
\begin{align}
 \sum_{i> n}y_{i}=-\gamma_0 y_n-\tfrac{1}{2}y_n,\quad \forall n\geq 2.\label{PONACT:5}
\end{align}
Thus
\begin{align}
 \sum_{i> n+1}y_{i}=-\gamma_0 y_{n+1}-\tfrac{1}{2}y_{n+1},
 \quad \forall n\geq 1.\label{PONACT:6}
\end{align}
Subtracting \eqref{PONACT:6} from \eqref{PONACT:5}, we obtain
\begin{align}
y_{n+1}=(-\gamma_0 -\tfrac{1}{2})(y_n-y_{n+1}),\quad \forall n\geq 2.\label{PONACT:7}
\end{align}
Since $\gamma_0>0$,  by \eqref{PONACT:7}, we have
\begin{align}
y_{n+1}\frac{\gamma_0 -\tfrac{1}{2}}{\gamma_0+\tfrac{1}{2}}=y_n,
\quad \forall n\geq 2.\label{PONACT:8}
\end{align}
Now we claim that
\begin{align}
y_n=0,\quad \forall n\geq2.\label{PONACT:9}
\end{align}
Suppose to the contrary that there exists $i_0\geq2$ such that
\begin{align}y_{i_0}\neq0.\label{PONACT:09}
\end{align}
Then by \eqref{PONACT:8}, we have
$y_{i_0}=y_{i_0+1}\frac{\gamma_0 -\tfrac{1}{2}}{\gamma_0+\tfrac{1}{2}}$.
Thus,
\begin{align}
\gamma_0\neq\frac{1}{2}.\label{PONACT:10}
\end{align}
Then by \eqref{PONACT:8}, we have
\begin{align}
y_{n+1}=\frac{\gamma_0+\tfrac{1}{2}}{\gamma_0 -\tfrac{1}{2}}y_n,
\quad \forall n\geq 2.\label{PONACT:08}
\end{align}
Set $\alpha:=\frac{\gamma_0+\tfrac{1}{2}}{\gamma_0 -\tfrac{1}{2}}$. Then  by $\gamma_0>0$ again,
\begin{align}
|\alpha|>1.\label{PONACT:11}
\end{align}
By \eqref{PONACT:08} and Fact~\ref{FE:1}, we have
$\sum_{i> 2}y_i=y_2\sum_{i\geq1}\alpha^i$ and
the former series is convergent.
Thus \eqref{PONACT:11} implies that
$y_2=0$ and then $y_n=0,\forall n>2$ by \eqref{PONACT:08},
 which contradicts \eqref{PONACT:09}.
Hence \eqref{PONACT:9} holds. Then by Fact~\ref{FE:1},
\begin{align}
y^*=(\tfrac{1}{2}y_1,0,0,\ldots,0,\ldots).
\end{align}
Hence $\|y^*\|\leq\tfrac{1}{2}\|y\|$ and thus \eqref{PONACT:3} holds.
Then by the assumption, we have
\begin{align*}
\|y^*\|&\leq\tfrac{1}{2}\|y\|\leq\tfrac{1}{2}\big(\|x\|+\|y\|
+\|x^*+y^*\|+\|x-y\|\cdot\|y^*\|\big)\\
&\leq j(\|x\|+\|y\|+\|x^*+y^*\|+\|x-y\|\cdot\|y^*\|).
\end{align*}
Hence the implication \ref{BCREX:E3} holds.

\ref{BCREX:E4}: By Fact~\ref{f:Fitz} and Fact~\ref{FE:1}\ref{NEC:2},
 $\pos F_{A^*}=\gra A^*$ and $\pos F_{N_C}=\gra N_C$.
Then directly apply \ref{BCPONA:E01}\&\ref{BCREX:E3} and
Fact~\ref{BCPCon}.
\end{proof}

\begin{remark}
\label{r:muchharder}
Example~\ref{FCPEX:3} provides  a negative answer to  Problem \ref{prob2} as asserted.
\end{remark}

\begin{remark}
\label{r:final}
It is not as easy to find a counterexample to Problem~\ref{prob2} as it is for Problem~\ref{prob1}. Indeed,   Fact~\ref{BCPCon} and Fact~\ref{F4} imply that, to find a counterexample,
we need to start with two maximally monotone operators $A,B:X\rightrightarrows X^*$  such that $A+B$
 is maximally monotone but it does not satisfy the well known sufficient transversality
  condition for the maximal monotonicity of the sum operator
   in a reflexive space \cite[Lemma 5.1]{SiZ} and \cite[Lemma~5.8]{BWY7}, that is:
 \begin{align}\label{cq}\bigcup_{\lambda>0}\lambda\left[\dom A-\dom B\right]\ \text{ is a closed subspace of $X$}.
 \end{align}
Otherwise, \eqref{cq} ensures that \eqref{Probcon} in Problem~\ref{prob2} holds by Fact~\ref{F4} and \cite[Lemma~5.8]{BWY7}.

Finally, as we mentioned in Section~\ref{sec:bc}  an affirmative answer to Problem \ref{prob2}
 would rederive Simons' theorem
(Fact~\ref{BCPCon}). Indeed, Simons  \cite[Corollary~34.5]{Si2} shows in  detail
 how to deduce the classic result of Br\'{e}zis, Crandall and Pazy \cite{BreCA} from his result.
\end{remark}

 \vfill

\paragraph{Acknowledgments.} Heinz Bauschke was partially supported by the
Natural Sciences and Engineering Research Council of Canada and by
the Canada Research Chair Program. Jonathan  Borwein was partially
supported by the Australian Research  Council. Xianfu Wang was
partially supported by the Natural Sciences and Engineering Research
Council of Canada.

\end{document}